\documentclass{amsart}

\usepackage{latexsym,amssymb,amsmath,bm,enumitem,verbatim,todonotes,stmaryrd}
\usepackage[hidelinks]{hyperref}

\makeatletter
\@namedef{subjclassname@2020}{%
  \textup{2020} Mathematics Subject Classification}
\makeatother

\title[The Ordering Principle and DC]{The Ordering Principle and Dependent Choice}

\usepackage{bbm}

\DeclareMathOperator{\HOD}{HOD}

\DeclareMathOperator{\Ord}{Ord}

\DeclareMathOperator{\dom}{dom}

\DeclareMathOperator{\id}{id}

\DeclareMathOperator{\forces}{\Vdash}

\DeclareMathOperator{\supp}{supp}

\DeclareMathOperator{\CH}{CH}
\DeclareMathOperator{\ZFC}{ZFC}
\DeclareMathOperator{\ZF}{ZF}

\newcommand{\Coll}{Coll}

\DeclareMathOperator{\PP}{\mathbb P}

\newcommand{\MM}{\mathcal{M}}
\newcommand{\NN}{\mathcal{N}}

\newcommand{\DC}{\mathrm{DC}}
\newcommand{\AC}{\mathrm{AC}}
\DeclareMathOperator{\fix}{\mathrm{fix}}
\newcommand{\G}{\mathcal G}
\newcommand{\F}{\mathcal F}

\newcommand{\SSS}{\mathcal S}
\newcommand{\PPP}{\mathbb P}

\newcommand{\GB}{\mathrm{GB}}
\newcommand{\HS}{\mathrm{HS}}
\newcommand{\sym}{\mathrm{sym}}

\newcommand{\QQ}{\mathbb Q}
\newcommand{\sG}{\mathcal{G}}
\newcommand{\sS}{\mathcal{S}}
\newcommand{\sF}{\mathcal{F}}
\newcommand{\sE}{\mathcal{E}}
\newcommand{\sH}{\mathcal{H}}
\newcommand{\sT}{\mathcal{T}}
\newcommand{\OP}{\mathrm{OP}}

\DeclareMathOperator{\ran}{ran}

\author{Peter Holy}
\address{Institut f\"ur diskrete Mathematik und Geometrie\\
TU Wien\\
Wiedner Hauptstrasse 8-10/104\\
1040 Vienna\\
Austria}
\email{peter.holy@tuwien.ac.at}

\author{Jonathan Schilhan}
\address{University of Vienna\\
Institute of Mathematics\\
Kurt Gödel Research Center\\
Kolingasse 14-16\\
1090 Vienna\\
Austria}
\email{jonathan.schilhan@univie.ac.at}


\begin{document}

\begin{abstract}
  We introduce finite support iterations of symmetric systems, and use them to provide a strongly modernized proof of David Pincus' classical result that the axiom of dependent choice is independent over $\ZF$ with the ordering principle together with a failure of the axiom of choice. 
\end{abstract}

\maketitle

\newtheorem{fact}{Fact}
\newtheorem{lemma}[fact]{Lemma}
\newtheorem{theorem}[fact]{Theorem}
\newtheorem{corollary}[fact]{Corollary}
\newtheorem{claim}[fact]{Claim}
\newtheorem{subclaim}[fact]{Subclaim}
\newtheorem{conjecture}[fact]{Conjecture}
\newtheorem{observation}[fact]{Observation}
\newtheorem{proposition}[fact]{Proposition}
\newtheorem{counterexample}[fact]{Counterexample}
\theoremstyle{definition}
\newtheorem{question}[fact]{Question}
\newtheorem{remark}[fact]{Remark}
\newtheorem{definition}[fact]{Definition}

\section{Introduction}

The ordering principle, $\OP$, is the statement that every set can be linearly ordered.
The axiom of choice, $\AC$, in one of its equivalent forms, states that every set can be wellordered, and thus clearly implies $\OP$. That this implication cannot be reversed was shown by Halpern and L\'evy (see \cite[Section 5.5]{jech}): The argument proceeds by showing that the basic Cohen model, which is well-known to satisfy $\ZF+\lnot\AC$ (see \cite[Section 5.3]{jech}), satisfies $\OP$. This model is obtained by first forcing to add countably many Cohen reals, and then passing to a symmetric submodel of this extension, in which we still have the set of those Cohen reals, but no well-ordering of it. We may informally say that we \emph{forget about} the wellordering in this submodel. On the other hand, it is easy to see that this submodel already fails to satisfy the axiom of dependent choice $\DC$: The generic set of Cohen reals that were added is clearly infinite, however Dedekind-finite (see \cite[Exercise~5.18]{jech}), i.e., $\omega$ does not inject into it. It is well-known (and an easy exercise) that $\DC$ implies the notions of being finite and of being Dedekind-finite to coincide.

The goal of our paper is to provide a new and strongly modernized proof of the following classical result of Pincus \cite{pincus}:

\begin{theorem}[Pincus]\cite{pincus}\label{th:main}
  $\DC$ is independent over $\ZF+\OP+\lnot\AC$.
\end{theorem}

Given the properties of the basic Cohen model that we reviewed above, this amounts to verifying the relative consistency of $\ZF+\OP+\DC+\lnot\AC$, starting from $\ZF$. 

Pincus' paper makes use of the ramified forcing notation which developed directly out of Cohen's independence proof for $\CH$. This old-fashioned way of presenting forcing already became outdated and essentially obsolete by the time \cite{pincus} was published (see e.g. Shoenfield's \cite{Shoenfield}) and therefore, while he provides a very nice outline of his arguments, from a modern point of view, the details in his paper are difficult to grasp. For this reason, we think that providing a modern and essentially self-contained account of his result will be very interesting for and helpful to the set theoretic community. Furthermore, this paper provides an application of the technique of symmetric iterations that has been initiated by Asaf Karagila and has been further developed with the second author in \cite{symext}. Although finite support iterations have already appeared in some form in Karagila's \cite{KARAGILA_2019}, we provide a more compact and, at least in our view, simpler approach that follows more closely the familiar notation from usual forcing iterations. It is our hope that the presentation we give below can at some point become part of the general folklore, just as is the case for forcing. For this to happen, the method must be applicable in a broad sense, and our article is a proof-of-concept for this. 

\medskip

Our basic line of argument towards Pincus' result essentially follows the outline at the beginning of \cite[Section 1]{pincus}: The starting point is the basic Cohen model. It has a set of Cohen reals with no wellordering, and in fact, no countably infinite subset. In order to resurrect $\DC$, one simply adds a surjection from $\omega$ onto this set, thus, however, even resurrecting $\AC$. So what Pincus actually does is he adds not just one, but countably many such surjections, and then passes to a symmetric submodel of this second extension in which he forgets about the ordering of this set of surjections, thereby obtaining a new failure of $\DC$. The failure of $\DC$ is thus shifted to a higher level of complexity. Now, the idea is to continue this process for $\omega_1$-many stages, and that any possible failure of $\DC$ will somehow appear at an intermediate stage, and will actually be fixed by the very next stage, i.e., $\DC$ will hold in the final model. What we will actually show however, is something slightly stronger, namely that $\sigma$-covering holds between our symmetric extension and the corresponding full forcing extension, and that this in turn implies $\DC$ to hold in our symmetric extension. Finally, modifying the standard arguments for the basic Cohen model, it is still possible to show that $\OP$ holds in our symmetric extension, while $\AC$ fails in it.

\medskip

Using this as a basic guideline, instead of following any of the further details in Pincus' paper, we came up with our own arguments, which we will provide below. 

\medskip

This paper is organised as follows: In Section \ref{section:preliminaries}, we will introduce some basic terminology regarding symmetric systems and extensions. In Section \ref{section:gb}, we will briefly comment on how to deal with second order models in the context of symmetric extensions. In Section \ref{section:iterations}, we will introduce finite support iterations of symmetric systems. In Section \ref{section:extension}, we will formally introduce the symmetric iteration that we will use to produce our desired model. In Section \ref{section:o}, we will verify that the ordering principle holds in our symmetric extension. In Section \ref{section:dc}, we will show that $\DC$ holds in our symmetric extension, thus establishing Theorem \ref{th:main}.

\section{Preliminaries}\label{section:preliminaries}

A symmetric system is a triple of the form $\SSS=(\PP,\G,\F)$, where $(\PP,\le)$ is a preorder, $\G$ is a group of automorphisms of $\PP$ and $\F$ is a normal filter on the set of subgroups of $\G$. If $\dot x$ is a $\PP$-name and $\pi\in\G$, we inductively let $$\pi(\dot x)=\{(\pi(p),\pi(\dot y))\mid(p, \dot y)\in\dot x\},$$ and we let $\sym(\dot x)=\{\pi\in\G\mid\pi(\dot x)=\dot x\}$. The following fact is well-known and used throughout the paper: 

\begin{fact}
    Let $\dot x_0, \dots, \dot x_n$ be $\PP$-names, $\pi$ an automorphism of $\PP$, $p \in \PP$ and $\varphi(v_0, \dots, v_n)$ some formula in the language of set theory. Then $$p \Vdash_{\PP} \varphi(\dot x_0, \dots, \dot x_n) \text{ if and only if } \pi(p) \Vdash_{\PP} \varphi(\pi(\dot x_0), \dots, \pi(\dot x_n)).$$
\end{fact}

We say that a $\PP$-name $\dot x$ is \emph{symmetric} (according to $\sS$) if $\sym(\dot x)\in\F$. We (inductively) say that $\dot x$ is \emph{hereditarily symmetric} (according to $\sS$) if $\dot x$ is symmetric and whenever $(p, \dot y) \in \dot x$, $\dot y$ is hereditarily symmetric. Given a symmetric system $\SSS$, we let $\HS_{\SSS}$ denote the collection of hereditarily symmetric $\PP$-names, and we omit the subscript $\SSS$ when it is clear from context. We also refer to elements of $\HS$ as \emph{$\SSS$-names}. The class $\HS$ is further stratified into sets $\HS_\alpha=\HS_{\SSS,\alpha}$, $\alpha$ an ordinal, consisting of the $\sS$-names of rank $<\alpha$.

\medskip

An interesting fact is that we do not have to require our generic filters to be \emph{fully generic} for $\mathbb{P}$: satisfying the weaker property of being \emph{symmetrically generic} suffices. This will briefly be useful in Section~\ref{section:iterations}, but isn't necessary for understanding the main result. 

\begin{definition}
    We say that a dense set $D \subseteq \mathbb{P}$ is \emph{symmetric} if $\sym(D) = \{ \pi \in \sG : \pi[D] = D\} \in \sF$.
    Let $G$ be a filter on $\mathbb{P}$. Then $G$ is \emph{$\sS$-generic}, or \emph{symmetrically generic}, over $M$, if for every symmetric dense subset $D \subseteq \mathbb{P}$ in $M$, $G \cap D \neq \emptyset$. 
\end{definition}

If $M$ is a transitive model of $\ZF$, a symmetric extension of $M$ via $\sS$, or in other words, an $\sS$-generic extension of $M$, is a model of the form $M[G]_\SSS=\{\dot x^G\mid\dot x\in\HS\}$ for an $\SSS$-generic filter $G$ over $M$. We let $\forces_\SSS$ denote the (symmetric) forcing relation of the system $\SSS$, which is defined inductively just like the usual forcing relation, however restricting to hereditarily symmetric names (in particular also in the existential and universal quantification steps). The forcing theorem holds with respect to this relation and $\SSS$-generic extensions \cite{symext}.

\begin{fact}[see \cite{symext}]
Let $G$ be $\sS$-generic over $M$. Then we have the following: 
\begin{itemize}
    \item \emph{The symmetric forcing theorem}: The forcing theorem holds with respect to $\Vdash_\sS$ and $\sS$-generic extensions of $M$.
    \item $M[G]_\sS \models \ZF$.
    \item There is a $\mathbb{P}$-generic $H$ over $M$ so that $M[H]_\sS = M[G]_\sS$.
\end{itemize}  
\end{fact}

The last item really says that there is no distinction between the models obtained via full versus symmetric generics.

\medskip

If $A$ is a set of $\PP$-names, then $A^\bullet = \{\mathbbm{1}\} \times A$ is the canonical $\PP$-name for the set containing the elements of $A$ (or more precisely, their evaluations by the generic filter). Similar notation is applied to sequences of $\PP$-names, so that for instance $\langle \dot a_i : i < n \rangle^\bullet$ becomes the canonical name for the ordered tuple of $\dot a_i$, $i < n$.

\medskip

We will sometimes use the following general fact, which says that we can uniformly find names for definable objects.

\begin{fact}\label{fact:namebydef}
Let $\sS= (\PP,\G,\F)$ be a symmetric system, and let $\varphi(u,v_0, \dots, v_n)$ be a formula in the language of set theory. Then, there is a definable class function $F$ so that for any $\sS$-names $\dot x_0, \dots, \dot x_n$ and $p\in \PP$ with $$p \Vdash_{\sS} \exists! y \varphi(y, \dot x_0, \dots, \dot x_n),$$ $\dot y=F(p, \dot x_0, \dots, \dot x_n)$ is an $\sS$-name with $\bigcap_{i \leq n} \sym(\dot x_i) \leq \sym(\dot y)$ so that $$p \Vdash_{\sS} \varphi(\dot y, \dot x_0, \dots, \dot x_n).$$
\end{fact} 

\begin{proof}
Let $\gamma$ be the least ordinal such that \[p \Vdash_{\sS} \exists y \in \HS_\gamma^\bullet\ \varphi(y, \dot x_0, \dots, \dot x_n).\] Let $\dot y$ be the set of all pairs $(q, \dot z) \in \PP \times \HS_\gamma$ so that $$q \Vdash \forall y  (\varphi(y, \dot x_0, \dots, \dot x_n) \rightarrow \dot z \in y) \}.$$
\end{proof}

\section{Symmetric extensions as second order models}\label{section:gb}

Let $\SSS=(\PP,\G,\F)$ be a symmetric system in a model $\MM=(M,\mathcal C)$ of $\GB$, that is G\"odel-Bernays set theory, with $M$ its domain of sets, and with $\mathcal C$ its domain of classes. In case we are starting with a model of $\ZF$, it yields a model of $\GB$ when endowed with its definable classes. In $\MM$, we say that a class $\dot X\subseteq\PP\times\HS$ is a \emph{class $\SSS$-name} if \[\sym(\dot X):=\{\pi\in\G\mid\pi(\dot X)=\dot X\}\in\F,\] where $\pi(\dot X) = \{ (\pi(p), \pi(\dot x)) : (p, \dot x) \in \dot X \}$. Let $G$ be an $\sS$-generic filter over~$\MM$, and let $\NN=\MM[G]_{\SSS}=(M[G]_\sS,\mathcal C[G]_\sS)$, where $\mathcal C[G]_\sS$ is obtained by evaluating class $\sS$-names in $\mathcal C$ with $G$.\footnote{Note that the classes of $\NN$ thus include all classes of $\MM$ (for they have canonical symmetric names).} We will use uppercase letters to refer to classes and class names, while lowercase letters indicate sets or set names. When we allow for classes as parameters in first order formulas, we also mean to include additional atomic formulas of the form $x\in X$.

\begin{proposition}\label{proposition:classes}
   The symmetric forcing theorem can be extended to first order formulas using classes as parameters, and $\NN\models\GB$.
\end{proposition}
\begin{proof}
  The verification of the extension of the symmetric forcing theorem is very much standard (proceeding exactly as for the usual forcing theorem) - see also \cite{symext}. Let us verify that the axiom of Collection holds in $\NN$. So assume $\varphi$ is a first order formula using class parameters, and $p\in\PP$ is such that $p\forces_{\sS}\forall x\,\exists y\,\varphi(x,y,\vec X)$, with $\vec X$ a finite sequence of class $\sS$-names, and let $\dot a\in M$ be a $\PP$-name. Since we may code any finite number of classes by a single class, it suffices to consider a single class $\sS$-name $\dot X\in\mathcal C$, which we may also assume to code the set parameters appearing in~$\varphi$. Let
      $\dot z=\{(q,\dot y)\mid \exists\dot x\in\ran(\dot a)\ \dot y\textrm{ is of minimal rank s.t.  }q\forces_{\sS}\varphi(x,\dot y,\dot X)\}.$
 Then, $\sym(\dot z)\ge\sym(\dot X)\cap\sym(\dot a)\in\F$, and $p\forces_{\sS}\forall x\in\dot a\,\exists y\in\dot z\,\varphi(x,y,\dot X)$, thus witnessing Collection to hold in $\NN$.
  Comprehension, and Class Comprehension, that is, the closure of $\mathcal C[G]$ under definability, are verified in similar ways (and somewhat more easily). We thus leave the details here to our readers.
\end{proof}

\section{Finite support iterations of symmetric systems}\label{section:iterations}

\begin{definition}[Two-step iteration, see e.g.\,\cite{symext}]\label{def:twostep}
Let $\sS = (\mathbb{P}, \sG, \sF)$ be a symmetric system and $\dot \sT = (\dot{\mathbb{Q}}, \dot \sH, \dot \sE)^\bullet$ be an $\sS$-name for a symmetric system, where $\sym(\dot{\sT}) = \sG$. Then, we define the two-step iteration $\sS * \dot \sT = (\mathbb{P} *_{\sS} \dot{\mathbb{Q}}, \sG *_\sS \dot \sH, \sF *_\sS \dot \sE )$, where 

\begin{enumerate}
\item $\mathbb{P} *_{\sS} \dot{\mathbb{Q}}$ consists of all pairs $(p, \dot q)$, where $p \in \mathbb{P}$ and $\dot q$ is an $\sS$-name for an element of $\dot{\mathbb{Q}}$, together with the usual order on $\mathbb{P}* \dot{\mathbb{Q}}$,\footnote{Note that this is a dense subset of the usual forcing iteration $\mathbb{P} * \dot{\mathbb{Q}}$.}
\item $\sG *_\sS \dot \sH$ consists of all pairs $(\pi, \dot \sigma)$, where $\pi \in \sG$ and $\dot \sigma$ is an $\sS$-name for an element of $\dot \sH$, and $(\pi, \dot \sigma)$ is identified with the map $$(p, \dot q) \mapsto (\pi(p), \dot \sigma(\pi(\dot q))),$$ 
\item $\sF *_\sS \dot \sE$ is generated by all groups of the form $(H_0, \dot H_1)$, where $H_0 \in \sF$ and $\dot H_1$ is an $\sS$-name for an element of $\dot \sE$ with $H_0 \leq \sym(\dot H_1)$, and $(H_0, \dot H_1)$ is identified with $\{ (\pi, \dot \sigma) : \pi \in H_0, \Vdash_\sS \dot \sigma \in \dot H_1 \}$.
\end{enumerate}
\end{definition}

A few remarks have to be made about this definition. The fact that we identify pairs with other types of objects should not lead to any confusion. When we write $\dot \sigma(\pi(\dot q))$, what we mean is a particular $\sS$-name for the result of $\dot \sigma$ applied to $\pi(\dot q)$.\footnote{$\pi(\dot q)$ denotes the usual application of $\pi$ to the $\sS$-name $\dot q$. Note that since $\sym(\dot{\mathbb{Q}}) = \sG$, $\pi(\dot q)$ is again an $\sS$-name for an element of $\dot{\mathbb{Q}}$.} While there is in fact a way to uniformly choose such a name (see \cite{symext}), the easiest way to make sense of this, and what we will actually do, is to simply identify a pair $(p, \dot q)$ with the set of equivalent conditions $\{ (p, \dot r) : \dot r\in\HS_\gamma,\,\Vdash_\sS \dot q = \dot r \}$, where $\gamma$ is least so that this set is nonempty. This has the added advantage that $\mathbb{P} *_{\sS} \dot{\mathbb{Q}}$ really becomes a set, while technically, there are proper class many possible names for elements of $\dot{\mathbb{Q}}$. Again, this identification makes no difference in practice.

It is then relatively straightforward to check that $(\pi, \dot{\sigma})$ preserves the order on $\mathbb{P} *_{\sS} \dot{\mathbb{Q}}$. One can compute (\cite[proof of Lemma 3.2]{symext}) that $$(\pi_0, \dot \sigma_0) \circ (\pi_1, \dot \sigma_1) = (\pi_0 \circ \pi_1,\dot\sigma_0 \circ \pi_0(\dot \sigma_1) ),$$ and that $$(\pi, \dot \sigma)^{-1} = (\pi^{-1}, \pi^{-1}(\dot\sigma^{-1})),$$ where $\dot\sigma_0 \circ \pi_0(\dot \sigma_1)$ and $\dot\sigma^{-1}$ are $\sS$-names for the respective objects. In particular, $(\pi, \dot{\sigma})$ is an automorphism, and $\sG *_\sS \dot \sH$ forms a group. 

Now, it is possible to make sense of (3), as it can be checked that the sets $(H_0, \dot H_1)$ are subgroups of $\sG *_\sS \dot \sH = (\sG, \dot \sH)$. It turns out that the filter $\sF *_\sS \dot \sE$ generated by these subgroups is in fact a normal filter. This is a bit more tricky to prove and, letting $\bar\pi=(\pi_0,\dot\pi_1)$, it is generally not the case that $\bar \pi (H_0, \dot H_1) \bar \pi^{-1}$ is a group of the form $(K_0, \dot K_1)$ as in (3) again. On the other hand, if $H_0 \leq \sym(\pi_0^{-1}(\dot \pi_1^{-1}))$, which we may achieve by shrinking $H_0$, then $\bar \pi (H_0, \dot H_1) \bar \pi^{-1} = (\pi_0 H_0 \pi_0^{-1}, \dot \pi_1 \pi_0(\dot H_1) \dot \pi_1^{-1})$ -- see \cite[proof of Lemma 3.2]{symext}.

\begin{lemma}[{\cite[Lemma 3.2]{symext}}]
$\sS * \dot \sT$ is a symmetric system. 
\end{lemma}

Moreover, one can prove a factorization theorem (\cite[Theorem 3.3]{symext}) that expresses precisely that an extension via $\sS * \dot \sT$ is of the form $\MM[G]_\sS[H]_{\sT}$ and vice-versa. There is some care to be taken though: It \emph{does not follow} from $G$'s $\PP$-genericity over $\MM$ and $H$'s $\dot \QQ^G$-genericity over $\MM[G]_\sS$, that $G * H = \{ (p, \dot q) \in \PP *_\sS \dot \QQ : p \in G \wedge \dot q^G \in H \}$ is itself $\PP *_\sS \dot \QQ$-generic over $\MM$. Rather, the factorization theorem states that if $G$ is $\sS$-generic (that is, symmetrically generic) over $\MM$ and $H$ is $\dot \sT^G$-generic over $\MM[G]_\sS$, then $G * H$ is also $\sS * \dot \sT$-generic over $\MM$. Conversely, if $K$ is some $\sS * \dot \sT$-generic over $\MM$, then $K = G * H$, where $G = \dom K$ is $\sS$-generic over $\MM$ and $H = \{ \dot q^G : \dot q \in \ran(K) \}$ is $\dot\sT^G$-generic over $\MM[G]_\sS$. In either case, $\MM[G]_\sS[H]_{\dot\sT^G} = \MM[G * H]_{\sS * \dot \sT}$. 

On the level of the models alone, there is no difference between those obtained via full generics or those obtained via symmetric generics. Thus, it is nevertheless the case than an $\sS * \dot \sT$-extension obtained via a full generic is exactly the result of extending in succession using full generics of the respective systems, and vice-versa.

\begin{definition}[Finite support iteration]\label{def:fsi}
Let $\delta$ be an ordinal. Let $$\langle \sS_{\alpha}, \dot \sT_\alpha : \alpha < \delta \rangle = \langle (\mathbb{P}_{\alpha}, \sG_\alpha, \sF_\alpha), (\dot{\mathbb{Q}}_\alpha, \dot \sH_\alpha, \dot \sE_\alpha)^\bullet : \alpha < \delta \rangle,$$ be such that each $\sS_\alpha$ is a symmetric system, $\dot \sT_\alpha$ is an $\sS_\alpha$-name for a symmetric system and $\sym(\dot \sT_\alpha) = \sG_\alpha$. Then we call this sequence a \emph{finite support iteration} of length $\delta$ if for each $\alpha < \delta$:

    \begin{enumerate}
       \item  \begin{enumerate}
        \item $\mathbb{P}_\alpha$ consists of sequences $\bar p = \langle \dot p(\beta) : \beta < \alpha \rangle$, where $\bar p \restriction \beta \in \mathbb{P}_\beta$ and $\dot p(\beta)$ is an $\sS_\beta$ name for an element of $\dot{\mathbb{Q}}_\beta$, for all $\beta < \alpha$.
        \item $\sG_\alpha$ consists of automorphisms of $\mathbb{P}_\alpha$ represented, as detailed below, by sequences $\bar \pi = \langle \dot \pi(\beta) : \beta < \alpha \rangle$, where $\bar \pi \restriction \beta \in \sG_\beta$ and $\dot \pi(\beta)$ is an $\sS_\beta$ name for an element of $\dot \sH_\beta$, for all $\beta < \alpha$.
        \item $\sF_\alpha$ is generated by subgroups of $\sG_\alpha$ represented, as detailed below, by sequences $\bar H = \langle \dot H(\beta) : \beta < \alpha \rangle$,  where $\bar H \restriction \beta \in \sF_\beta$, $\dot H(\beta)$ is an $\sS_\beta$ name for an element of $\dot \sE_\beta$ and $\bar H \restriction \beta \leq \sym_{\sS_\beta}(\dot H(\beta))$, for all $\beta < \alpha$.
        \end{enumerate}
        \item $\sS_{\alpha +1} = \sS_\alpha * \dot{\sT_\alpha}$, where pairs $(\bar p, \dot q)$, $(\bar \pi, \dot \sigma)$, $(\bar H, \dot K)$ as in Definition~\ref{def:twostep} are identified with sequences $\bar p^\frown \dot q$, $\bar \pi^\frown \dot \sigma$ and $\bar H^\frown \dot K$ respectively.
      
      \end{enumerate}
      For $\alpha < \delta$ limit, 
      \begin{enumerate}\setcounter{enumi}{2}
       \item \begin{enumerate}
        \item $\mathbb{P}_\alpha$ consists exactly of those $\bar p$ as above, so that $\Vdash_{\sS_\beta} \dot p(\beta) = \mathbbm{1}$ for all but finitely many $\beta < \alpha$, and $\bar q \leq \bar p$ iff $\bar q \restriction \beta \leq \bar p \restriction \beta$ for each $\beta < \alpha$,
        \item $\sG_\alpha$ consists exactly of those $\bar \pi$ as above, so that $\Vdash_{\sS_\beta} \dot \pi(\beta) = \id$ for all but finitely many $\beta < \alpha$, and $$\bar \pi(\bar p) = \bigcup_{\beta < \alpha} (\bar \pi \restriction \beta) (\bar p \restriction \beta),$$
        \item $\sF_\alpha$ is generated by the subgroups of the form $\bar H = \langle \dot H(\beta) : \beta < \alpha \rangle$ as above, where $\Vdash_{\sS_\beta} \dot H(\beta) = \dot \sH_\beta$ for all but finitely many $\beta < \alpha$, and $\bar \pi \in \bar H$ iff $\bar \pi \restriction \beta \in \bar H \restriction \beta$ for all $\beta < \alpha$.
        \end{enumerate}
    \end{enumerate}
\end{definition}

As it stands, the above is just a definition. But of course, the way to read this in practice is as an instruction on how to recursively construct a finite support iteration. The definition says precisely what to do in each step of the construction. In the successor step, when constructing $\sS_{\alpha +1} = \sS_\alpha * \dot{\sT_\alpha}$, we already know that this results in a symmetric system. On the other hand, to ensure that such a construction always makes sense we still need to check that if the limit step is defined as in (3), we really do obtain a symmetric system again. 

Before doing so, let's make a few simple observations about finite-support iterations. Let $\supp(\bar p) := \{ \alpha : \neg\Vdash_{\sS_\alpha} \dot p(\alpha) = \mathbbm{1}\}$, $\supp(\bar \pi) := \{ \alpha : \neg\Vdash_{\sS_\alpha} \dot \pi(\alpha) = \id\}$ and $\supp(\bar H) := \{ \alpha : \neg\Vdash_{\sS_\alpha} \dot H(\alpha) = \dot\sH_\alpha\}$. We refer to the above as \emph{support} in each case.

\begin{lemma}\label{lem:fsilemma}
Let $\langle \sS_{\alpha}, \dot \sT_\alpha : \alpha < \delta \rangle$ be a finite support iteration as above, $\alpha < \delta$ arbitrary. 
\begin{enumerate}
\item All $\bar p \in \mathbb{P}_\alpha$, $\bar \pi \in \sG_\alpha$, $\bar H \in \sF_\alpha$ have finite support.
    \item For any $\bar p \in \mathbb{P}_\alpha$, $\bar \pi \in \sG_\alpha$, $\supp(\bar p) = \supp(\bar \pi(\bar p))$.
    \item For any $\bar p \in \mathbb{P}_\alpha$, $\bar \pi \in \sG_\alpha$, $\beta \leq \alpha$, $\bar \pi(\bar p) \restriction \beta = (\bar\pi \restriction \beta)(\bar p \restriction \beta)$.
    \item For any $\beta \leq \alpha$, $\bar p \in \mathbb{P}_\beta$, $\bar \pi \in \sG_\beta$ and $\bar H \in \sF_\beta$, we have that $\bar p^\frown \langle \dot{\mathbbm{1}}_\gamma : \gamma \in [\beta, \alpha) \rangle \in \mathbb{P}_\alpha$, $\bar \pi^\frown \langle \dot{\id}_\gamma : \gamma \in [\beta, \alpha) \rangle \in \sG_\alpha$, $\bar H^\frown \langle \dot{\sH}_\gamma : \gamma \in [\beta, \alpha) \rangle \in \sF_\alpha$, where $\dot{\mathbbm{1}}_\gamma$ is a name for the trivial condition of $\dot{\mathbb{Q}}_\gamma$ and $\dot{\id}_\gamma$ for the identity in $\dot{\sH}_\gamma$. In particular, $\mathbb{P}_\beta = \{ \bar p \restriction \beta : \bar p \in \mathbb{P}_\alpha \}$.
    \item For any $\bar \pi, \bar \sigma \in \sG_\alpha$, $\bar \pi \circ \bar \sigma = \langle \dot \pi(\beta) \circ (\bar \pi \restriction \beta)(\dot \sigma(\beta)) : \beta < \alpha\rangle$ and $\supp(\bar \pi \circ \bar \sigma) \subseteq \supp(\bar \pi) \cup \supp(\bar \sigma)$. In particular, $(\bar \pi \circ \bar \sigma) \restriction \beta = (\bar \pi \restriction \beta) \circ( \bar \sigma \restriction \beta)$, for all $\beta \leq \alpha$.
    \item For any $\bar \pi \in \sG_\alpha$, $\bar \pi ^{-1} = \langle (\bar\pi \restriction \beta)^{-1}(\dot \pi(\beta)^{-1}) : \beta < \alpha\rangle$ and $\supp(\bar \pi^{-1}) = \supp(\bar \pi)$. In particular, $(\bar \pi^{-1}) \restriction \beta = (\bar \pi \restriction \beta)^{-1}$, for all $\beta \leq \alpha$.
    \item For any of the generators $\bar H, \bar K \in \sF_\alpha$ as in (1)(c) in the definition of a finite support iteration, $\bar E = \langle\dot E(\beta) : \beta \leq \alpha \rangle \in \sF_\alpha$, where $\dot E(\beta)$ is a name for $\dot H(\beta) \cap \dot K(\beta)$ for each $\beta < \alpha$, $\supp(\bar E) \subseteq \supp(\bar H) \cup \supp(\bar K)$ and $\bar E \leq \bar H \cap \bar K$.
    \item For each $\bar p \in \PPP_\alpha$, $\bar \pi \in \sG_\alpha$ there are $\bar H, \bar K \in \sF_\alpha$ so that $\bar H \restriction \beta \leq \sym(\dot p(\beta))$ and $\bar K \restriction \beta \leq \sym((\bar\pi\restriction\beta)^{-1}(\dot \pi(\beta)^{-1}))$ for each $\beta < \alpha$.
    \item For any $\bar \pi \in \sG_\alpha$ and $\bar H \in \sF_\alpha$, where for each $\beta<\alpha$, $$\bar H \restriction \beta \leq \sym((\bar \pi \restriction \beta)^{-1}(\dot \pi(\beta)^{-1}) ),$$ we have that $$\bar \pi \bar H \bar \pi^{-1} = \langle \dot \pi(\beta) (\pi\restriction \beta)(\dot H(\beta)) \dot \pi(\beta)^{-1} : \beta < \alpha \rangle$$ and $\supp(\bar \pi \bar H \bar \pi^{-1}) = \supp(\bar H).$
    \item If $\langle \mathbb{P}'_\beta, \dot{\mathbb{Q}}_\beta : \beta \leq \alpha \rangle$ is the usual finite-support iteration of forcing notions, then $\mathbb{P}_\alpha$ is a dense subposet of $\mathbb{P}'_\alpha$.
\end{enumerate}
\end{lemma}

\begin{proof}
These are all straightforward inductions on $\alpha$. For (5) and (6), use the inverse and composition formulas we have already given for the two-step iteration. For (9) use the analogous statement for two-step iterations we have mentioned above. Note that since our conditions have finite supports, limit stages in (10) are trivial.
\end{proof}

\begin{lemma}
Let $\langle \sS_\beta , \dot{\sT_\beta} : \beta < \alpha \rangle$ be a finite-support iteration and let $\sS_\alpha = (\mathbb{P}_\alpha, \sG_\alpha, \sF_\alpha)$ be defined as in (3) of Definition~\ref{def:fsi}. Then, $\sS_\alpha$ is a symmetric system.
\end{lemma}

\begin{proof}
$\mathbb{P}_\alpha$ is clearly a forcing poset. Next, let $\bar \pi \in \sG_\alpha$ and $\bar p \in \mathbb{P}_\alpha$ be given. According to Item (3) of Lemma~\ref{lem:fsilemma}, $(\bar \pi \restriction \beta) (\bar p \restriction \beta) \subseteq (\bar \pi \restriction \gamma) (\bar p \restriction \gamma)$ for every $\beta \leq \gamma < \alpha$, so $\bar \pi(\bar p)$ is a sequence as in (1)(a) of Definition~\ref{def:fsi}. Items (1) and (2) of Lemma~\ref{lem:fsilemma} imply that $\supp(\bar \pi(\bar p))$ is still finite, so $\bar \pi(\bar p) \in \mathbb{P}_\alpha$. Clearly, $\bar \pi$ is also order-preserving and the inverse and composition formulas given in (5) and (6) above also work for the elements of $\sG_\alpha$. Thus, $\sG_\alpha$ is a group of automorphisms, and similarly, any $\bar H$ as in (3)(c) of Definition~\ref{def:fsi} is a subgroup of $\sG_\alpha$. It remains to check that $\sF_\alpha$ is normal. So let $\bar H \in \sF_\alpha$, $\bar \pi \in \sG_\alpha$ be arbitrary. $\supp(\bar \pi) \subseteq \beta$ for some $\beta \leq \alpha$ and using (7) and (8) above we can find $\bar K \in \sF_\beta$, $\bar K \leq \bar H \restriction \beta$, so that $\bar K \restriction \gamma \leq \sym((\bar\pi\restriction\gamma)^{-1}(\dot \pi(\gamma)^{-1}))$ for each $\gamma < \beta$. Then $\bar H' := \bar K^\frown \bar H \restriction [\beta, \alpha) \leq \bar H$ and $\bar H' \in \sF_\alpha$. From (9), also using (5), we can compute that $$\bar \pi \bar H' \bar \pi^{-1} = \langle \dot \pi(\gamma) (\pi\restriction \gamma)(\dot H'(\gamma)) \dot \pi(\gamma)^{-1} : \gamma < \alpha \rangle,$$ which has finite support and thus is in $\sF_\alpha$.
\end{proof}

\begin{lemma}\label{lem:equalautom}
Let $\langle \sS_{\alpha}, \dot \sT_\alpha : \alpha < \delta \rangle$ be a finite support iteration as above. Fix some $\alpha < \delta$, let $\bar \pi, \bar \sigma \in \sG_\alpha$, $\bar p \in \mathbb{P}_\alpha$, and assume that for all $\beta < \alpha$, $\bar p \restriction \beta \Vdash \dot \pi(\beta) = \dot \sigma(\beta).$ Then, for any $\mathbb{P}_\alpha$-name $\dot x$, $$\bar p \Vdash \bar \pi(\dot x) = \bar \sigma(\dot x).$$  
\end{lemma}

\begin{proof}
This is essentially the same as \cite[Lemma 5.5]{symext}. More precisely, work with a generic $G \ni \bar p$ and show by induction on $\beta\le\alpha$ that for any $\bar q$, $(\bar \pi\restriction\beta)(\bar q\restriction\beta) \in G$ iff $(\bar \sigma\restriction\beta)(\bar q \restriction\beta) \in G$. The rest then follows by induction on the rank of $\dot x$. 
\end{proof}

A factorization theorem can also be proven for finite support iterations, similarly to the one for two-step iterations. We will not need this anywhere in our results, so the reader may immediately skip to the next section, but it is still important enough for the general theory so that we would like to include it.

Suppose that $\langle \sS_{\alpha}, \dot \sT_\alpha : \alpha \leq \delta \rangle$ is a finite support iteration and $\alpha \leq \delta$ is fixed. By recursion on the length $\delta$, one defines an $\sS_{\alpha}$-name $\langle \dot \sS_{\alpha, \gamma}, \dot \sT_{\alpha, \gamma} : \gamma \in [\alpha, \delta] \rangle^\bullet$ for a finite support iteration that naturally corresponds to the tail of the iteration. Simultaneously, one defines for each $\sS_\delta$-name $\dot x$, an $\sS_\alpha$-name $[\dot x]_{\alpha, \delta}$ for an $\dot \sS_{\alpha, \delta}$-name, and similarly, for each $\sS_\alpha$-name $\dot y$ for an $\dot \sS_{\alpha, \delta}$-name, an $\sS_\delta$-name $]\dot y[_{\alpha, \delta}$. Further, for any $\bar p \in \mathbb{P}_\delta$, $\bar \pi \in \sG_\delta$ and $\bar H \in \sF_\delta$, one defines $\sS_\alpha$-names $[\bar p \restriction [\alpha, \delta)]$, $[\bar \pi \restriction [\alpha, \delta)]$ and $[\bar H \restriction [\alpha, \delta)]$ for respective objects in the system $\dot \sS_{\alpha, \delta}$. 

The recursive construction proceeds as follows: For $\delta = \alpha$, we let $\dot \sS_{\alpha, \gamma}$ be a name for the trivial system $(\{ \mathbbm{1}\}, \{\id \}, \{ \{\id \} \})$. $[\bar p \restriction [\alpha, \alpha)]$ is simply a name for $\mathbbm{1}$. At each step $\delta$, by recursion on the rank of names, we define $$[\dot x]_{\alpha, \delta} = \{ (\bar p \restriction \alpha, ([\bar p \restriction [\alpha, \delta)], [\dot z]_{\alpha, \delta})^\bullet) : (\bar p, \dot z) \in \dot x \},$$ and similarly, $$] \dot y[_{\alpha, \delta} = \{ (\bar p, ]\dot z[_{\alpha, \delta}) : \bar p \restriction \alpha \Vdash_{\sS_\alpha} ([\bar p \restriction [\alpha, \delta)], \dot z)^\bullet \in \dot y \}.$$

We let $\dot \sT_{\alpha, \delta} = [\dot \sT_\delta]_{\alpha, \delta}$. For $\delta = \gamma +1$, we define $$[\bar p \restriction [\alpha, \delta)] = ([\bar p \restriction [\alpha, \gamma)]^\frown [\dot p(\gamma)]_{\alpha, \gamma}))^\bullet,$$ and for $\delta$ limit, $$[\bar p \restriction [\alpha, \delta)] = \bigcup_{\gamma \in [\alpha, \delta)} [\bar p \restriction [\alpha, \gamma)].$$ Similarly for $[\bar \pi \restriction [\alpha, \delta)]$ and $[\bar H \restriction [\alpha, \delta)]$.

The properties claimed above are easily checked by induction on $\delta$. 

\begin{theorem}[Factorization for finite support iterations]
Whenever $G$ is $\sS_\alpha$-generic over $\MM$ and $H$ is $\dot \sS_{\alpha, \delta}^G$-generic over $\MM[G]_{\sS_\alpha}$, then $G * H = \{ \bar p : \bar p \restriction \alpha \in G \wedge [\bar p \restriction [\alpha, \delta)]^G \in H\}$ is $\sS_{\delta}$-generic over $\MM$. Similarly, whenever $K$ is $\sS_\delta$-generic over $\MM$, then $K = G * H$, where $G = \{ \bar p \restriction \alpha : \bar p \in K \}$ is $\sS_\alpha$-generic over $\MM$ and $H = \{ [\bar p \restriction [\alpha, \delta)]^G : \bar p \in K \}$ is $\dot \sS_{\alpha, \delta}^G$-generic over $\MM[G]_{\sS_\alpha}$.

In either case, $([\dot x]^G)^H = \dot x^{G * H}$ for every $\sS_\delta$-name $\dot x$ and $]\dot y[^{G*H} = (\dot y^G)^H$ for every $\sS_\alpha$-name $\dot y$ for an $\dot \sS_{\alpha, \delta}$-name. In particular, $\MM[G*H]_{\sS_\delta} = \MM[G]_{\sS_\alpha}[H]_{\dot \sS_{\alpha, \delta}^G}$.

\end{theorem}

\begin{proof}
    
    If $D \subseteq \mathbb{P}_\delta$ is open dense, show by induction on $\delta$ that the set $$[D]_{\alpha, \delta} = \{ (\bar p\restriction\alpha , [\bar p \restriction [\alpha, \delta)]) : \bar p \in D\}$$ is an $\sS_{\alpha}$-name for an open dense subset of the forcing $\dot{\mathbb{P}}_{\alpha, \delta}$ corresponding to $\dot \sS_{\alpha, \delta}$. Moreover, if $\bar H \leq \sym(D)$, then $\bar H \restriction \alpha \leq \sym([D]_{\alpha, \delta})$ and $\Vdash_{\sS_\alpha} [\bar H \restriction [\alpha, \delta)] \leq \sym( [D]_{\alpha, \delta})$. This is how we check that $G * H$ is $\sS_\delta$-generic, given the genericity of $G$ and $H$. 
    
    The other direction, starting from an $\sS_\delta$-generic and obtaining the genericity of $G$ and $H$, is completely analogous: From a name $\dot D$ for an open dense subset of $\dot{\mathbb{P}}_{\alpha, \delta}$, define $$]\dot D[_{\alpha, \delta} = \{ \bar p : \bar p \restriction \alpha \Vdash_{\sS_\alpha} [\bar p \restriction [\alpha, \delta)] \in \dot D\}.$$ Everything else is just as straightforward and similar to \cite[Theorem 3.3]{symext}.
\end{proof}

\section{The symmetric extension}\label{section:extension}

\begin{definition}
  Given a forcing notion $\mathbb{R}$, we let $\sT(\mathbb{R})=(\QQ,\sH,\sE)$ denote the symmetric system where 
  
  \begin{itemize}
  \item $\QQ$ is the finite support product of $\omega$-many copies of $\mathbb{R}$, i.e., $\mathbb{Q}$ consists of finite partial functions $p \colon \omega \to \mathbb{R}$ together with the extension relation given by $q \leq p$ iff $\dom p \subseteq \dom q$ and $$ \forall n \in \dom p\ q(n) \leq p(n),$$
  \item $\sH$ is the group of finitary permutations of $\omega$,\footnote{A permutation $\pi \colon \omega \to \omega$ is \emph{finitary} if $\pi(n) = n$ for all but finitely many $n \in \omega$.} where $\pi \in \sH$ acts on $\QQ$ coordinate-wise, i.e., $\dom \pi(q) = \pi''\dom q$ and $$\pi(q)(n) = q(\pi^{-1}(n))$$ for every $n \in \omega$,
  \item $\sE$ is generated by the subgroups of $\sH$ of the form $$\fix(e) = \{ \pi \in \sH : \forall n \in e\ \pi(n) = n\},$$ for $e \in [\omega]^{<\omega}$.
  \end{itemize}

\end{definition}

To see that $\sE$ is normal, simply note that for any $\pi\in\sH$, $\pi \fix(e) \pi^{-1} = \fix(\pi''e)$.
The system for the basic Cohen model (see \cite[Section 5.3]{jech}) is exactly $\sT(\mathbb{C})$, where $\mathbb{C}$ is Cohen forcing.

\medskip

We will construct a finite support symmetric iteration of the form $$\langle \sS_{\alpha}, \dot \sT_\alpha : 1 \leq \alpha \leq \omega_1 \rangle = \langle (\mathbb{P}_{\alpha}, \sG_\alpha, \sF_\alpha), (\dot{\mathbb{Q}}_\alpha, \dot \sH_\alpha, \dot \sE_\alpha)^\bullet : 1 \leq \alpha \leq \omega_1 \rangle,$$ where for each $\alpha < \omega_1$, $\dot \sT_\alpha$ is an $\sS_\alpha$-name for a symmetric system of the form $\sT(\dot{\mathbb{R}}_\alpha)$, where $\dot{\mathbb{R}}_\alpha$ is an $\sS_\alpha$-name for a forcing notion. We start by letting $\sS_1$ be the basic Cohen system, i.e., $\sS_1 = \sT(\mathbb{C})$ where $\mathbb{C}$ is Cohen forcing, and we will inductively define the remainder of our symmetric iteration.\footnote{We start the iteration at index $1$ rather than $0$ for notational convenience related to the coherence of the indexing at steps below $\omega$ and after. One could also start with letting $\sS_0$ be some trivial system and then ignore the first coordinate when writing $\bar p$, $\bar \pi$ or $\bar H$.}

Suppose we have constructed $\sS_\delta$, for some $\delta \leq \omega_1$. Before defining $\dot{\mathbb{R}}_\delta$ (in case $\delta < \omega_1$), we first verify some general properties about the iteration up to $\delta$.\footnote{We will only know precisely what $\sS_\delta$ is once we have specified what happens in each step, but the description we have given so far is sufficient to make some general observations.} To simplify notation, for the rest of this section, let us abbreviate $\sS = \sS_{\delta}$, $\mathbb{P}= \PP_{\delta}$, $\sG = \sG_{\delta}$, and $\sF = \sF_{\delta}$.

For $e \in [\delta \times \omega]^{<\omega}$, write $e_\alpha$ for the $\alpha^\textrm{th}$ section of $e$, i.e., $e_\alpha = \{ n : (\alpha, n) \in e \}$. Let $\fix(e) = \langle  \fix(e_\alpha)\check{\,} : \alpha < \delta\rangle \in \sF$. While $\sF$ contains more complicated groups, it usually suffices to only consider those of the form $\fix(e)$, as can be seen by the following:

\begin{lemma}\label{lem:simplefilter}
Let $\dot x \in \HS$, $\bar p \in \mathbb{P}$. Then there is $\bar q \leq \bar p$, $\dot y \in \HS$ and $e\in[\delta \times \omega]^{<\omega}$ so that $\fix(e) \leq \sym(\dot y)$ and $\bar q \Vdash \dot x = \dot y$.
\end{lemma}

\begin{proof}
Let $\bar H \leq \sym(\dot x)$, $\bar q \leq \bar p$, and $e \in [\delta \times \omega]^{<\omega}$, so that $$\bar q \restriction \alpha \Vdash_{\PP_\alpha} \dot H(\alpha) = \fix(e_\alpha)\check{\,}$$ for every $\alpha < \delta$. This is easy to achieve by extending $\bar p$ finitely often, deciding all $\dot H(\alpha)$ for $\alpha \in \supp(\bar H)$. Let $\gamma$ be least so that $\dot x \in \HS_\gamma$ and let $\dot y$ consist of all pairs $(\bar r, \dot z) \in \mathbb{P} \times \HS_\gamma$ so that $\bar r \restriction \alpha \Vdash_{\PP_\alpha} \dot H(\alpha) = \fix(e_\alpha)\check{\,}$ for all $\alpha < \delta$, and $\bar r \Vdash \dot z \in \dot x$. Clearly $\bar q \Vdash \dot x = \dot y$. 

\begin{claim}
$\fix(e) \leq \sym(\dot y)$. 
\end{claim}
\begin{proof}
Let $\bar \pi \in \fix(e)$, and let $(\bar r, \dot z) \in \dot y$ be arbitrary. Consider $\bar \sigma\in\sG$ where each $\dot \sigma(\alpha)$ is so that the following is forced: ``either $\dot H(\alpha) = \fix(e_\alpha)\check{\,}$ and $\dot \sigma(\alpha) = \dot \pi(\alpha)$, or $\dot H(\alpha) \neq \fix(e_\alpha)\check{\,}$ and $\dot \sigma(\alpha)$ is the identity". Such a name can be chosen uniformly by Fact~\ref{fact:namebydef}. Then, $\bar \sigma \in \bar H$, and we note that also $\bar \sigma(\bar r) \restriction \alpha \Vdash \dot H(\alpha) = \fix(e_\alpha)\check{\,}$, for each $\alpha$, as $\bar \sigma \restriction \alpha \in \bar H \restriction \alpha \leq \sym(\dot H(\alpha))$. Thus, also $\bar \sigma(\bar r) \restriction \alpha \Vdash \dot \pi(\alpha) = \dot \sigma(\alpha)$ for each $\alpha$. By Lemma~\ref{lem:equalautom}, we have that $\bar \sigma(\bar r) \Vdash \bar \pi(\dot z) = \bar \sigma(\dot z)$. In particular, $\bar \sigma(\bar r) \Vdash \bar \pi(\dot z) \in \dot x$. On the other hand, we may also verify by induction that the conditions $\bar \sigma(\bar r)$ and $\bar \pi(\bar r)$ are equivalent. 
This implies that $(\bar \pi(\bar r), \bar \pi(\dot z)) \in \dot y$, as desired. 
\end{proof}
\end{proof}

Something similar can be done on the level of the automorphisms themselves. Let $f$ be a function from $\delta$ to finitary permutations of $\omega$, so that for all but finitely many $\alpha < \delta$, $f(\alpha)$ is the identity. Then, we can consider $\tau_f = \langle \check{f}(\alpha) : \alpha < \delta \rangle \in \sG$. $f$ acts naturally on $\delta \times \omega$ via $$f \cdot (\alpha, n) = (\alpha, f(\alpha)(n)).$$ For any $e \in [\delta \times \omega]^{<\omega}$, $$\tau_f \fix(e) \tau_f^{-1} = \fix (f \cdot e),$$ where $f \cdot e = \{ f \cdot (\alpha, n) : (\alpha, n) \in e \}$.

\begin{lemma}\label{lem:simplegroup}
Let $\bar \pi \in \sG$, $\bar p \in \PP$, and let $\dot x$ be an arbitrary $\mathbb{P}$-name. Then, there is $\bar q \leq \bar p$ and $f$ as above such that $$\bar q \Vdash \bar \pi(\dot x) = \tau_f(\dot x).$$
Moreover, whenever $\bar \pi \in \fix(e)$, we can ensure that $\tau_f \in \fix(e)$ as well. 
\end{lemma}
\begin{proof}
Let $\bar q$ decide $\dot \pi(\alpha)$ for each $\alpha$ and then use Lemma~\ref{lem:equalautom}.
\end{proof}

This shows that we can in fact consider the simpler system $\sS' = (\PP, \sG', \sF')$ where $\sG'$ consists only of the automorphisms of the form $\tau_f$ and $\sF'$ is generated by only the groups $\fix(e) \cap \sG'$. Observe for instance that Lemma~\ref{lem:simplefilter} can be applied hereditarily to show that for any $\bar p$, and $\dot x \in \HS_\sS$, there is $\bar q \leq \bar p$ and $\dot y \in \HS_{\sS'}$ so that $\bar q \Vdash \dot x = \dot y$. 

This is an instance of a much more general situation in which we want to consider only particular names for the conditions, automorphisms and generators of the filter at each iterand of our symmetric iteration. In the language of \cite[Section 5]{symext}, $\sS'$ would be called a \emph{reduced iteration}. 
It won't be necessary for us to actually pass to $\sS'$, and we haven't even shown that $\sS'$ is a symmetric system, albeit this is easy to check. Rather, it will be sufficient to use the previous lemmas directly in order to simplify our arguments. 

\medskip

 For $(\alpha, n) \in \delta \times \omega$, let $\dot g_{\alpha, n}$ be a canonical name for the $\dot{\mathbb{R}}_\alpha$-generic added in the $n^\textrm{th}$ coordinate of $\dot\QQ_\alpha$. To be precise, let $\gamma_\alpha$ be least so that for any $\sS_\alpha$-name $\dot s$ for an element of $\dot{\mathbb{R}}_\alpha$, there is $\dot r \in \HS_{\sS_{\alpha}, \gamma_\alpha}$ with $\Vdash_{\sS_\alpha} \dot s = \dot r$. Let $$\dot g_{\alpha, n} = \{ (\bar p, \dot r) : \bar p\in\PP_{\alpha+1}\wedge\,\dot r \in \HS_{\sS_\alpha, \gamma_\alpha} \wedge \bar p \restriction \alpha\Vdash_{\sS_\alpha} \dot r = \dot p(\alpha)(n)\}.$$ 
Then, $\dot g_{\alpha, n} \in \HS$ and $\fix(\{\alpha, n\})\leq \sym(\dot g_{\alpha, n})$. More generally, note that if $\bar \pi \in \sG$ is such that $\Vdash \dot \pi(\alpha)(\check{n}) = \check{m}$, then we will have that $\bar \pi(\dot g_{\alpha, n}) = \dot g_{\alpha, m}$.

We define $$\dot A_\delta = \{ \bar \pi(\dot g_{\alpha,n}) : \bar  \pi \in \sG, (\alpha, n) \in \delta \times \omega \}^\bullet \in \HS.$$ 
Clearly, $\sym(\dot A_\delta)=\sG$, and by Lemma~\ref{lem:simplegroup}, $$\Vdash \dot A_\delta = \{ \dot g_{\alpha, n} : (\alpha, n) \in \delta \times \omega\}^\bullet.$$ 

Concluding our definition, if $\delta<\omega_1$, we let $\dot \sT_\delta = (\dot{\mathbb{Q}}_\delta, \dot \sH_\delta, \dot \sE_\delta)^\bullet$ be an $\sS_\delta$-name for $\sT(\Coll(\omega, \dot A_\delta))$: recall that for a set $A$, $\Coll(\omega, A)$ is the poset consisting of finite partial functions from $\omega$ to $A$ ordered by extension. As $\sym(\dot A_\delta) = \sG = \sG_\delta$, Fact~\ref{fact:namebydef} shows that we can indeed require that $\sym(\dot\sT_\delta) = \sG_\delta$.

\section{The ordering principle}\label{section:o}

In this section, we show that the ordering principle $\OP$ holds after performing the above-described symmetric iteration of length $\omega_1$ over a ground model of $\ZFC$ with a definable wellorder of its universe. For example, we may work over G\"odel's constructible universe. From now on, let $\sS = \sS_{\omega_1}$, $\mathbb{P}= \PP_{\omega_1}$, $\sG = \sG_{\omega_1}$, and $\sF = \sF_{\omega_1}$.\footnote{In fact, while this is not needed for our main result, note that the results of this section would hold for any $\delta\le\omega_1$ rather than just $\omega_1$.} 

\begin{lemma}\label{lem:restr}
    Let $\dot x_i \in \HS$, $\alpha<\omega_1$, $e_i \in [\alpha \times \omega]^{<\omega}$ and $\fix(e_i) \leq \sym(\dot x_i)$, for $i <n$. Let $\varphi(v_0, \dots, v_{n-1})$ be a formula with all free variables shown, and let $\bar p,\bar q\in\PP$. Then, whenever $\bar p \restriction \alpha = \bar q \restriction \alpha$, it holds that $$\bar p \Vdash_{\sS} \varphi(\dot x_0, \dots, \dot x_{n-1}) \text{ iff } \bar q \Vdash_{\sS} \varphi(\dot x_0, \dots, \dot x_{n-1}).$$
\end{lemma}

\begin{proof}
    Suppose $\bar p \Vdash_{\sS} \varphi(\dot x_0, \dots, \dot x_{n-1})$ but $\bar r \Vdash_{\sS} \neg\varphi(\dot x_0, \dots, \dot x_{n-1})$ for some $\bar r \leq \bar q$. It suffices to find $\bar \pi \in \bigcap_{i  < n} \fix(e_i)$ so that $\bar \pi(\bar r) \parallel \bar p$, to yield a contradiction. Simply let $\dot \pi(\beta)$ be a name for the identity for all $\beta < \alpha$, thus already ensuring that $\bar \pi \in \bigcap_{i  < n} \fix(e_i)$, and then define $\dot \pi(\beta)$, for $\beta \geq \alpha$ inductively as follows: When $\bar \pi \restriction \beta$ has been defined, simply let $\dot \pi(\beta)$ be a name for a finitary permutation of $\omega$ mapping the domain of $(\bar \pi \restriction \beta)(\dot r(\beta))$ away from the domain of $\dot p(\beta)$. If $a, b \subseteq \omega$ are finite, then a finitary permutation $\pi$ such that $\pi[a] \cap b = \emptyset$ can of course be easily defined from $a$ and $b$ as parameters. So Fact~\ref{fact:namebydef} shows that such an $\sS_\beta$-name $\dot \pi(\beta)$ exists.\footnote{Of course we are basically just showing that every tail of the iteration is homogeneous with respect to the group $\sG$.}
\end{proof}

Define $\dot \Gamma = \{\bar \pi(\dot G) : \bar \pi \in \sG \}^\bullet$, where $\dot G$ is the canonical name for the $\mathbb{P}$-generic filter. While $\dot \Gamma$ is not an $\sS$-name in general, it is still a symmetric $\PP$-name and it plays an important role in any symmetric system. The following is a quite general observation and shows that $M[G]_\sS = \HOD^{M[G]}_{M(A)\cup \{ \Gamma \}}$, where $\Gamma = \dot \Gamma^G$.\footnote{It is based on the fact that the names of the form $(\dot g_{\alpha_0, n_0}, \dots, \dot g_{\alpha_k, n_k})^\bullet$ for $(\alpha_i, n_i) \in \omega_1 \times \omega$ form a \emph{respect-basis} for $\sS$ (see more in \cite[Section 6.4]{symext}).}

\begin{lemma}\label{lem:define}
    Let $\dot x \in \HS$ and $e \in [\omega_1 \times \omega]^{<\omega}$ so that $\fix(e) \leq \sym(\dot x)$. Whenever $G$ is $\PP$-generic over $\MM$, $x = \dot x^G$ and $\Gamma = \dot \Gamma^G$, then $x$ is definable in $\MM[G]$ from elements of $V$, from $\Gamma$ and from $\dot g_{\alpha, n}^G$ for $(\alpha, n) \in e$, as the only parameters.
\end{lemma}

\begin{proof}
    In $\MM[G]$, define $y$ to consist exactly of those $z$ so that $z \in \dot x^H$ for some $H \in \Gamma$ with $\dot g_{\alpha, n}^H = \dot g_{\alpha, n}^G$ for all $(\alpha, n) \in e$. We claim that $x = y$. Clearly, $x \subseteq y$ as $G \in \Gamma$. Now suppose that $H \in \Gamma$ is arbitrary, so that $\dot g_{\alpha, n}^H = \dot g_{\alpha, n}^G$, for all $(\alpha, n) \in e$. Then $H = \bar \pi(\dot G)^G$, for some $\bar \pi \in \sG$, and further, by Lemma~\ref{lem:simplegroup}, $H = \tau_f(\dot G)^G$ for some $f$. We obtain that $\dot g_{\alpha, n}^H = \tau_f(\dot g_{\alpha, n})^G = \dot g_{\alpha, n}^G$, for each $(\alpha, n) \in e$. But this is only possible if $\tau_f \in \fix(e)$. So also $\dot x^H = \tau_f(\dot x)^G = \dot x^G$ and we are done. 
\end{proof}

The following is very specific to the way we chose $\dot \sT_\alpha$: 

\begin{lemma}\label{lem:section}
    Let $\bar p \in \PP$, $\dot x \in \HS$ and $e \in [\omega_1 \times \omega]^{<\omega}$ be non-empty with $\fix(e) \leq \sym(\dot x)$. Further, let $\alpha = \max \dom(e)$. Then there is $\bar q \leq \bar p$ and $\dot y \in \HS$ with $\fix(\{\alpha \} \times e_\alpha) \leq \sym(\dot y)$ so that $\bar q \Vdash \dot y = \dot x$.
\end{lemma}

\begin{proof}
    By the previous lemma, for any generic $G$, $\dot x^G$ is definable in $\MM[G]$ from $\dot \Gamma^G$, $\langle \dot g_{\beta, n}^G : (\beta, n) \in e \rangle$ and parameters in $V$. But note that each $\dot g_{\beta, n}^G$, for $\beta < \alpha$, is itself definable from any $\dot g_{\alpha, m}^G$, as the latter enumerate $\dot A_\alpha^G$. Thus, $\dot x^G$ is already definable from $\dot \Gamma^G$, $\langle \dot g_{\alpha, n}^G : n \in e_\alpha \rangle$ and parameters in $V$. So we can find $\bar q \leq \bar p$ and a formula $\varphi$ so that $$\bar q \Vdash_{\mathbb{P}} \dot x = \{ z : \varphi(z, \dot \Gamma, \langle \dot g_{\alpha, n} : n \in e_\alpha \rangle^\bullet, \check{v}_0, \dots, \check{v}_k ) \},$$ for some $v_0, \dots, v_k \in \MM$. For some large enough $\gamma$, define $$\dot y = \{ (\bar r, \dot z) \in \PP \times \HS_\gamma : \bar r \Vdash \varphi(\dot z, \dot \Gamma, \langle \dot g_{\alpha, n} : n \in e_\alpha \rangle^\bullet, \check{v}_0, \dots, \check{v}_k ) \}.$$ We obtain that $\fix(\{ \alpha \} \times e_\alpha) \leq \sym(\dot y)$ and $\bar q \Vdash \dot y = \dot x$.  
\end{proof}

\begin{lemma}\label{lem:intersect}
  Let $\alpha<\omega_1$, $\bar p\in\PP$, $a_0, a_1 \in [\omega]^{<\omega}$ and $\dot x,\dot y\in\HS$ such that
  \begin{enumerate}
      \item $\bar p\forces\dot x=\dot y$,
      \item $\fix(\{\alpha\} \times a_0)\le\sym(\dot x)$,
      \item $\fix(\{\alpha\} \times a_1)\le\sym(\dot y)$.
  \end{enumerate}
  Then, there is $\bar q \leq \bar p$, $e \in [(\alpha +1) \times \omega]^{<\omega}$ and $\dot z\in\HS$ so that $e_\alpha = a_0 \cap a_1$, $\fix( e) \leq \sym(\dot z)$ and $\bar q\forces\dot z=\dot x$. 
\end{lemma}

\begin{proof}
 First, applying Lemma~\ref{lem:simplefilter}, we find $\bar q \leq \bar p$ such that $\fix (e') \leq \sym(\dot q(\alpha))$ for some $e' \in [\alpha \times \omega]^{<\omega}$. We define $e = e' \cup \{\alpha \} \times (a_0 \cap a_1)$. Next, instead of $\dot y$, let us consider $$\dot y' = \{ (\bar r, \tau) : \exists (\bar s, \tau) \in \dot y \, (\bar r \leq \bar q, \bar s)  \},$$
  and note that, $\bar q\forces\dot y'=\dot y = \dot x$. 
  We let $$\dot z = \bigcup_{\bar\pi\in\fix(e)} \bar\pi(\dot y').$$
  Now clearly, $\fix(e)\le\sym(\dot z)$. We claim that $\bar q \Vdash \dot z = \dot x$. Towards this end, let $G$ be an arbitrary generic containing $\bar q$. As $\dot y' = \id(\dot y') \subseteq \dot z$, we have that $\dot x^G = \dot y^G = \dot y'^G \subseteq \dot z^G$. To see that $\dot z^G \subseteq \dot x^G$, we show that $\bar \pi(\dot y')^G \subseteq \dot x^G$, for every $\bar\pi \in \fix(e)$. So fix $\bar \pi \in \sG$ now. According to Lemma~\ref{lem:simplegroup}, there is $f$ so that $\bar \pi(\dot y')^G = \tau_f(\dot y')^G$ and $\tau_f \in \fix(e)$.
  
  If $\tau_f(\bar q) \notin G$, clearly $\tau_f(\dot y')^G = \emptyset \subseteq \dot x^G$, as every condition appearing in a pair in $\tau_f(\dot y')$ is below $\tau_f(\bar q)$.
  
  So assume that $\tau_f(\bar q) \in G$. Consider for a moment the $\QQ_\alpha$-generic $H$ over $\mathcal{M}[G_\alpha]$ given by $G$, where $G_\alpha = \{ \bar r \restriction \alpha : \bar r \in G \}$. More precisely, $$H = \{ \dot s^{G_\alpha} : \exists \bar r \in G ( \dot r(\alpha) = \dot s) \}.$$ We have that $s = \dot q(\alpha)^{G_\alpha} \in H$ and moreover, as $\tau_f \restriction \alpha \in \sym(\dot q(\alpha))$, \begin{align*}
      f(\alpha)(s) &= f(\alpha)(\dot q(\alpha)^{G_\alpha}) \\ &= f(\alpha)\left((\tau_f \restriction \alpha)(\dot q(\alpha))^{G_\alpha}\right)\\ &= \tau_f(\bar q)(\alpha)^{G_\alpha}\in H.
  \end{align*} Let $d = \{ n \in \omega : f(\alpha)(n) \neq n  \} \cup a_0 \cup \dom(s)$, which is finite. By a density argument over $\MM[G_\alpha]$, we can find a finitary permutation $\sigma$ of $\omega$ that switches $a_1 \setminus a_0$ with a set disjoint from $d$, leaves everything else fixed, and is such that $\sigma(s) \in H$. In particular, $\sigma$ fixes $a_0$. Moreover, note that $f(\alpha)(\sigma(s)) \in H$ as well: $\sigma(s) \restriction \sigma[a_1 \setminus a_0]$ is not moved by $f(\alpha)$, and $\sigma(s) \restriction (\omega \setminus \sigma[a_1 \setminus a_0]) \subseteq s$, where we know that $f(\alpha)(s) \in H$.

  Back in $\MM$, let $h(\beta) = \id$ for every $\beta \in \omega_1 \setminus \{ \alpha \}$ and let $h(\alpha) = \sigma$. Then, $\tau_h \in \fix(\{ \alpha \} \times a_0)$ and $$\tau_{h} (\bar q) \restriction ({\alpha +1}), \tau_f(\tau_{h}(\bar q)) \restriction ({\alpha +1}) \in G_{\alpha+1} = \{ \bar r \restriction (\alpha +1) : \bar r \in G \}.$$
  
  Then we obtain that $\tau_h(\bar q) \Vdash \dot x = \tau_h(\dot y') = \tau_h(\dot y)$. By Lemma~\ref{lem:restr}, this is already forced by $\tau_h(\bar q) \restriction (\alpha +1)^\frown \langle \dot{\mathbbm{1}}_\beta : \beta \in [\alpha +1, \omega_1)\rangle
  \in G$. Thus, $\dot x^G = \tau_h(\dot y')^G = \tau_h(\dot y)^G$.

  Also, we note that $\tau_f \in \fix(\{\alpha \} \times ((a_0 \cap a_1) \cup \sigma[a_1 \setminus a_0])) \leq \sym(\tau_h(\dot y))$ (see the paragraph after Lemma~\ref{lem:simplefilter}). Hence, $\tau_f(\tau_h(\dot y)) = \tau_h(\dot y)$ and $\tau_f(\tau_h(\bar q)) \Vdash \tau_f(\dot x) =  \tau_h(\dot y)$. Similarly to before, this implies that $\tau_f(\dot x)^G = \tau_h(\dot y)^G$.
  Since $\tau_f(\bar q) \Vdash \tau_f(\dot x) = \tau_f(\dot y')$ and $\tau_f(\bar q) \in G$, we have $\tau_f(\dot x)^G = \tau_f(\dot y')^G$. So finally, we obtain that $$\dot x^G = \tau_h(\dot y)^G = \tau_f(\dot x)^G = \tau_f(\dot y')^G,$$ which is what we wanted to show.
\end{proof}

\begin{definition}
    Let $\bar p \in \PP$, $\dot x \in \HS$, $\alpha <\omega_1$ and $a \in [\omega]^{<\omega}$ be so that $\fix(\{\alpha \} \times a) \leq \sym(\dot x)$. We say that $\alpha$ is a \emph{minimal index for $\dot x$ below $\bar p$} if for any $\beta < \alpha$, $a' \in [\omega]^{<\omega}$ and $\dot z \in \HS$ with $\fix(\{\beta\} \times a') \leq \dot z$, $$\bar p \Vdash \dot x \neq \dot z.$$ 
    
    We say that \emph{$\{ \alpha\} \times a$ is a minimal support for $\dot x$ below $\bar p$} if $\alpha$ is a minimal index for $\dot x$ below $\bar p$ and for any $\bar q \leq \bar p$, $a' \in [\omega]^{<\omega}$  and $\dot z \in \HS$ with $\fix(\{\alpha\} \times a') \leq \sym(\dot z)$, if $\bar q \Vdash \dot x = \dot z$, then $a \subseteq a'$.
\end{definition}

\begin{corollary}\label{cor:minsup}
    For any $\dot x \in \HS$ and $\bar p \in \PP$, there is $\bar q \leq \bar p$, $\dot y \in \HS$, $\alpha < \omega_1$ and $a \in [\omega]^{<\omega}$ so that $\bar q \Vdash \dot x = \dot y$ and $\{\alpha \} \times a$ is a minimal support for $\dot y$ below $\bar q$.
\end{corollary}

\begin{proof}
From Lemma~\ref{lem:section}, there is a minimal $\alpha$ where we can find $\bar q \leq \bar p$, $\dot y$ and $a$ so that $\bar q \Vdash \dot y = \dot x$ and $\fix(\{ \alpha \} \times a) \leq \sym(\dot y)$. In that case, $\alpha$ is a minimal index for $\dot y$ below $\bar q$. Moreover then, fixing that minimal $\alpha$, there is a $\subseteq$-minimal $a$ for which we find $\bar q$, $\dot y$ as above. We claim that $\{ \alpha \} \times a$ is a minimal support for $\dot y$ below $\bar q$. Otherwise, there are $\bar q' \leq \bar q$, $a'$ and $\dot z$ with $\fix(\{\alpha \} \times a')\le\sym(\dot z)$ so that $\bar q' \Vdash \dot y = \dot z$ but $a \not\subseteq a'$. In particular $a \cap a'$ is a strict subset of $a$. According to Lemma~\ref{lem:intersect}, there is $\bar q'' \leq \bar q'$, $e \in [(\alpha +1) \times \omega]^{<\omega}$, $e_\alpha = a \cap a'$, and $\dot z'$ with $\fix(e) \leq \dot z'$, so that $\bar q'' \Vdash \dot z' = \dot z = \dot y$. If $\alpha = 0$, then $a$ was not $\subseteq$-minimal. If $\alpha > 0$, then $a \cap a' \neq \emptyset$ since otherwise $\alpha$ was not minimal, by Lemma~\ref{lem:section}. But then again, according to Lemma~\ref{lem:section}, $a$ was not $\subseteq$-minimal -- contradiction.
\end{proof}

Note that in the above corollary, neither $\alpha$ nor the set $a$ is necessarily unique. However, what is easily seen to be true using Lemma \ref{lem:intersect} is that if $\{\alpha\}\times a$ and $\{\beta\}\times b$ both are minimal supports for $\dot y$ below the same condition $\bar q$, then $\alpha = \beta$ and $a=b$.

\begin{lemma}\label{lem:Alin}
  There is an $\sS$-name $\dot <$ for a linear order of $\dot A$, such that $\sym(\dot <) = \sG$.
\end{lemma}
\begin{proof} In any model of $\ZF$, we can consider the definable sequence of sets $\langle X_\alpha : \alpha \in \Ord\rangle$, obtained recursively by setting $X_0 = \omega$, $X_{\alpha + 1} = (X_\alpha)^{\omega}$ and $X_\alpha = \bigcup_{\beta < \alpha} X_\beta$ for limit $\alpha$. We can recursively define linear orders $<_\alpha$ on $X_\alpha$, by letting $<_0$ be the natural order on $\omega$, $<_{\alpha+1}$ be the lexicographic ordering on $X_{\alpha +1}$ obtained from $<_\alpha$ and for limit $\alpha$, $x <_\alpha y$ iff, for $\beta$ least such that $x \in X_\beta$, either $y \notin X_\beta$ or $x <_\beta y$. Then $<_{\omega_1}$ is a definable linear order of $X_{\omega_1}$. Identifying the $\mathbb{R}_\alpha$-generics with the surjections they induce, note that $\dot A$ is forced to be contained in $X_{\omega_1}$, and by Fact~\ref{fact:namebydef}, there is an $\sS$-name $\dot <$ as required. 
\end{proof}

\begin{proposition}
  There is a class $\sS$-name $\dot F$ for an injection of the symmetric extension into $\Ord \times \dot A^{<\omega}$ such that $\sym(\dot F) = \sG$. In particular, $\OP$ holds in our symmetric extension.
\end{proposition}
  \begin{proof}
    Fix a global well-order $\vartriangleleft$ of $\MM$ and let $G$ be $\PP$-generic over $\MM$. We will first provide a definition of an injection $F$ in the full $\PP$-generic extension $\MM[G]$. Then, we will observe that all the parameters in this definition have symmetric names, which will let us directly build an $\sS$-name for $F$.
    
    For each $\alpha < \omega_1$, $a \in [\omega]^{<\omega}$ and each enumeration $h = \langle n_i : i < k \rangle$ of $a$, define $\dot G_{\alpha, a} = \{ \dot g_{\alpha, n} : n\in a \}^\bullet$, and $\dot t_{\alpha, h} = \langle \dot g_{\alpha, n_i} : i < k \rangle^\bullet$.  Let $\Gamma = \dot \Gamma^G$, $< = \dot <^G$ and $A = \dot A^G$. Given $x \in \MM[G]_\sS$, $F(x)$ will be found as follows: 
    
     First, let $(\bar p, \dot z, \alpha, a, h) \in \MM$ be $\vartriangleleft$-minimal with the following properties: 
     
     \begin{enumerate}
         \item (in $\MM$) $\{\alpha\} \times a$ is a minimal support for $\dot z$ below $\bar p$,
         \item (in $\MM$) $h$ is an enumeration of $a$ so that $\bar p$ forces that $\dot t_{\alpha, h}$ enumerates $\dot G_{\alpha, a}$ in the order of $\dot <$,
         \item there is $H \in \Gamma$, with $\bar p \in H$ and $\dot z^H = x$
     \end{enumerate} Such a tuple certainly exists by Corollary~\ref{cor:minsup} and since $G \in \Gamma$.
     \begin{claim}
     For any $H, K \in \Gamma$ with $\bar p \in H, K$, the following are equivalent:
    \begin{enumerate}[label=(\alph*)]
        \item $(\dot t_{\alpha, h})^H = (\dot t_{\alpha,h})^K$,
        \item $\dot z^H = \dot z^K$.
    \end{enumerate}
    \end{claim}
    \begin{proof}
    Let $H, K \in \Gamma$, $\bar p \in H, K$. $H$ is itself a $\PP$-generic filter over $\MM$ and $\dot \Gamma^H = \dot \Gamma^G = \Gamma$, as can be easily checked. Thus, there is $\bar \pi \in \sG$ so that $K = \bar \pi(\dot G)^H$. By Lemma~\ref{lem:simplegroup}, there is $f$ so that $K = \tau_f(\dot G)^H$. Now note that $\tau_f(\dot G)^H = \tau_f^{-1}[H]$ and $(\dot t_{\alpha, h})^K = (\dot t_{\alpha, h})^{\tau_f^{-1}[H]} = \tau_f(\dot t_{\alpha, h})^H$. Similarly, $\dot z^K = \tau_f(\dot z)^H$.
    
    Suppose that $(\dot t_{\alpha, h})^H = (\dot t_{\alpha, h})^K$. Then $(\dot t_{\alpha, h})^H = \tau_f(\dot t_{\alpha, h})^H$. The only way this is possible is if $f(\alpha)(n)= n$ for every $n \in a$. In other words, $\tau_f \in \fix(\{\alpha \} \times a)$. Thus $\dot z^H = \tau_f(\dot z)^H = \dot z^K$.

    Now suppose that $\dot z^H = \dot z^K = \tau_f(\dot z)^H$. We have that $\fix(f \cdot (\{\alpha \} \times a)) = \tau_f \fix(\{\alpha \} \times a) \tau_f^{-1} \leq \sym(\tau_f(\dot z))$. Since $\{\alpha\} \times a$ is a minimal support of $\dot z$ below $\bar p \in H$, it follows that $a \subseteq f(\alpha)[a]$ and by a cardinality argument, $a = f(\alpha)[a]$ . This also means that $\dot G_{\alpha, a} = \tau_f(\dot G_{\alpha, a})$. As $\bar p$ forces that $\dot t_{\alpha, h}$ is the $\dot <$-enumeration of $\dot G_{\alpha,a}$, we have that $\tau_f(\bar p)$ forces that $\tau_f( \dot t_{\alpha, h})$ is the $\tau_f(\dot <)$ enumeration of $\tau_f(\dot G_{\alpha,a})$. We have that $\bar p \in K = \tau_f^{-1}[H]$, so $\tau_f(\bar p) \in H$ and indeed $(\dot t_{\alpha, h})^K = \tau_f( \dot t_{\alpha, h})^H$ is the enumeration of $\tau_f(\dot G_{\alpha,a})^H = \dot G_{\alpha,a}^H$ according to $\tau_f(\dot <)^H = <$, which is exactly what $(\dot t_{\alpha, h})^H$ is.
    \end{proof}
      
By the claim, there is a unique $t \in A^{<\omega}$ so that $t = (\dot t_{\alpha, h})^H$, for some, or equivalently all, $H \in \Gamma$ with $\bar p \in H$ and $\dot z^H = x$. We let $F(x) = (\xi, t)$, where $(\bar p, \dot z, \alpha, a, h)$ is the $\xi^\textrm{th}$ element of $\MM$ according to $\vartriangleleft$. To see that this is an injection, assume that $x$ and $y$ both yield the same $(\bar p, \dot z, \alpha, a, h)$ and $t$. Let $H,K\in\Gamma$ with $\bar p \in H, K$, and with $\dot z^H = x$, $\dot z^K = y$. By definition $ t = (\dot t_{\alpha, h})^H = (\dot t_{\alpha, h})^K$ and according to the claim $x = \dot z^H = \dot z^K = y$.  This finishes the definition of $F$. 

The definition we have just given can be rephrased as $$\text{ $F(x) = y$ iff $\varphi(x, y, \Gamma, <)$},$$ where $\varphi$ is a first order formula using the class parameters $\Gamma$ and $<$, and the only parameters that are not shown are parameters from $\MM$, such as the class $\vartriangleleft$ or the class of $(\bar p, \dot z, \alpha, a, h)$ so that (1) and (2) hold. Simply let $$\dot F = \{ (\bar p, (\dot x, \dot y)^\bullet) : \dot x, \dot y \in \HS \wedge \bar p \Vdash_{\PP} \varphi(\dot x, \dot y, \dot \Gamma, \dot <) \},$$ where the parameters from $\MM$ in $\varphi$ are replaced by their check-names. Then, $\dot F\subseteq\PP\times\HS$, and $\sym(\dot F) = \sG$, so $\dot F$ is a class $\sS$-name, as desired. \end{proof}

\section{The axiom of dependent choice}\label{section:dc}

In this section, we will show that the axiom of dependent choice $\DC$ holds in our above symmetric extension. We will use the well-known (and easy to verify) fact that $\DC$ holds if and only if any tree without maximal nodes contains an increasing chain of length $\omega$.

\begin{lemma}
For each $\alpha < \omega_1$, $\Vdash_{\PP_\alpha} \dot \QQ_\alpha \text{ is countable}$. In particular, $\mathbb{P}$ is ccc.
\end{lemma}

\begin{proof}
 If $G$ is $\PP_\alpha$-generic, $\dot A_\alpha^G = \{ \dot g_{\beta,n}^G : (\beta, n) \in \alpha \times \omega \}$ is clearly countable in $\MM[G]$. In particular, $\Coll(\omega, \dot A_\alpha^G)$ and $\dot \QQ_\alpha^G$ are countable forcing notions. 
 
 By Lemma~\ref{lem:fsilemma}, Item (10), $\PP$ is just a dense subposet of the usual finite support iteration of the $\dot \QQ_\alpha$, and must be ccc.
\end{proof}

\begin{proposition}
Let $G$ be $\mathbb{P}$-generic over $\MM$. Then $\sigma$-covering holds between $\MM[G]_\sS$ and $\MM[G]$. That is, whenever $x \in \MM[G]$ is so that $\MM[G] \models \vert x \vert = \omega$ and $x \subseteq \MM[G]_\sS$, there is $y \in \MM[G]_\sS$ so that $\MM[G]_\sS \models \vert y \vert = \omega$ and $x \subseteq y$.
\end{proposition}

\begin{proof}
Let $x = \dot x^G$ for some $\mathbb{P}$-name $\dot x \in \MM$. For some $p \in G$ and large $\gamma$, $p \Vdash \dot x \subseteq \HS_\gamma^\bullet \wedge \dot x \text{ is countable}$. Using the ccc of $\PP$, we can find a countable set $c \subseteq \HS_\gamma$ so that $p \Vdash \dot x \subseteq c^\bullet$. Moreover, using Lemma~\ref{lem:simplefilter}, we can assume that for each $\dot z \in c$, there is $e \in [\omega_1\times \omega]^{<\omega}$ so that $\fix(e) \leq \sym(\dot z)$. Let $\alpha< \omega_1$ be large enough so that for each $\dot z \in c$ there is such $e$ in $[\alpha \times \omega]^{<\omega}$.

Recall Lemma~\ref{lem:define} and its proof: If $\fix(e) \leq \sym(\dot z)$, and $\langle (\beta_i, n_i) : i < k \rangle$ enumerates $e$, then $y \in \dot z^G$ iff $$\MM[G] \models \varphi(y, \dot z, \langle \dot g_{\beta_i, n_i} : i < k \rangle, \dot \Gamma^G, \langle \dot g_{\beta_i, n_i}^G : i < k \rangle),$$ 
for the formula $\varphi$ expressing that $y \in \dot z^H$ for some $H \in \dot \Gamma^G$ satisfying $\dot g_{\beta_i, n_i}^H = \dot g_{\beta_i, n_i}^G$ for all $i<k$. Since $\dot g_{\alpha, 0}^G$ enumerates $\dot A_\alpha^G$, there is a sequence $\langle m_i : i < k \rangle$ so that $\langle \dot g_{\beta_i, n_i}^G : i < k \rangle = \langle \dot g_{\alpha, 0}^G(m_i) : i  <k \rangle$. 

For any $\dot z \in c$, any $s \in (\alpha \times \omega)^{<\omega}$, and any $t \in \omega^{<\omega}$ of the same length, define $\dot x_{\dot z, s, t}$ to consist of all $(\bar p, \dot y) \in \PP\times \HS_\gamma$ so that $$\bar p \Vdash  \varphi(\dot y, (\dot z)\check{}, \langle \dot g_{s(i)} : i < \vert s \vert \rangle\check{}, \dot \Gamma, \langle \dot g_{\alpha, 0}(t(i)) : i < \vert t \vert \rangle^\bullet)$$ Then $\dot x_{\dot z, s, t}$ is an $\sS$-name with $\fix(\{(\alpha, 0)\}) \leq \sym(\dot x_{\dot z, s, t})$. Letting $h \in \MM$ be a surjection from $\omega$ to $\bigcup_{k\in\omega}c\times (\alpha \times \omega)^k \times \omega^k$, we find that $$ \dot d = \{ (n, \dot x_{h(n)})^\bullet : n \in \omega \}$$ is an $\sS$-name for a function with domain $\omega$ and with $x \subseteq \ran(\dot d^G)$, as desired 
\end{proof}

\begin{corollary}[see also \cite{KaragilaSchilhan2025}]
$\Vdash_\sS \DC$. 
\end{corollary}
\begin{proof}
Consider a generic $G$ and let $T \in \MM[G]_\sS$ be a tree without maximal nodes. Since $\MM[G] \models \AC$, there is an increasing chain $\langle t_n : n \in \omega \rangle$ of $T$ in $\MM[G]$. By the previous proposition, there is a countable $Y$ in $\MM[G]_\sS$ so that $\{t_n : n \in \omega\} \subseteq Y$. We may assume that $Y \subseteq T$. Recursively applying a pruning derivative to $Y$, whereby we remove all maximal elements in each step, we obtain a subtree $T' \subseteq Y$ without maximal elements. None of the $t_n$ could ever have been removed, so $T'$ is non-empty as e.g. $t_0 \in T'$. As $T'$ is countable, and thus well-ordered, we do find an increasing chain $\langle s_n : n \in \omega \rangle$ of $T'$, and thus also of $T$, in $\MM[G]_\sS$. 
\end{proof}

\begin{proposition}\label{proposition:notdc}
  $\forces_\sS\neg\DC_{\omega_1}$, and thus in particular $\forces_\sS\neg\AC$.
\end{proposition}
\begin{proof}
  This is essentially the same argument that is used to show that $\AC$ fails in the basic Cohen model (see for example \cite[Lemma 5.15]{jech}). Assume for a contradiction that there is $\dot F\in\HS$ such that $\bar p \in\PPP$ forces that $\dot F$ is an injection from $\omega_1$ into $\{\dot g_{\alpha, n}\mid (\alpha, n)\in \omega_1 \times \omega\}^\bullet$ (clearly, one can construct such an injection under $\DC_{\omega_1}$). By Lemma~\ref{lem:simplefilter} we can assume that $\fix(e)\leq\sym(\dot F)$ for some $e\in [\omega_1 \times \omega]^{<\omega}$. Pick $\gamma<\omega_1$, $\bar q\le \bar p$ and $\alpha >\max(\dom(e))$ such that $\bar q\forces\dot F(\check{\gamma})=\dot g_{\alpha,0}$. By Lemma~\ref{lem:restr}, this is already forced by $\bar q' = \bar q \restriction (\alpha + 1)^\frown\langle \dot{\mathbbm{1}}_\beta : \beta \in [\alpha +1, \omega_1) \rangle$. Let $\dot \pi(\beta)$ be a name for the identity for each $\beta \in \omega_1 \setminus \{\alpha \}$ and let $\dot \pi(\alpha)$ be a name for the permutation of $\omega$ switching $0$ with the minimal $n > 0$ outside of the domain of $\dot q(\alpha)$. It suffices to note three things: 
  \begin{enumerate}
   \item $\bar \pi \in \fix(e)$,
    \item $\Vdash \bar \pi(\dot g_{\alpha,0}) \ne \dot g_{\alpha,0}$, and
    \item $\bar \pi(\bar q')\parallel \bar q'$.
  \end{enumerate}
  This clearly poses a contradiction, as $\bar \pi(\bar q') \Vdash \dot F(\check{\gamma}) \neq \dot g_{\alpha, 0}$ while a compatible condition, $\bar q'$, forces the opposite. 
  \end{proof}
  
\subsection*{Acknowledgements}
The second author was supported in part by a UKRI Future Leaders Fellowship [MR/T021705/2]. This research was funded in whole or in part by the Austrian Science Fund (FWF) [10.55776/ESP5711024]. For open access purposes, the authors have applied a CC BY public copyright license to any author-accepted manuscript version arising from this submission. No data are associated with this article.
 
\bibliographystyle{plain}

\end{document}